\documentclass[12pt]{amsart}
\usepackage{amsmath}
\usepackage{amsfonts}
\usepackage{hyperref}
\usepackage{amssymb}
\usepackage{latexsym}
\usepackage{amsthm}
\usepackage{a4wide}

\usepackage{graphicx}

\usepackage[ansinew]{inputenc}
\usepackage{amsfonts,epsfig}
\usepackage{mathabx}

\newtheorem{theorem}{Theorem}

\newtheorem{coro}[theorem]{Corollary}

\newtheorem{proposition}[theorem]{Proposition}
\newtheorem{rk}[theorem]{Remark}
\newcounter{other}            

\theoremstyle{definition}

\DeclareMathAccent{\widecheck}{0}{mathx}{"71}

\makeatletter
\newcommand*{\rom}[1]{\expandafter\@slowromancap\romannumeral #1@}
\makeatother

\numberwithin{equation}{section}
\begin{document}

\title[Optimal domain of integration operators]{Recovering Hardy spaces from optimal domains of integration operators}

\author[Setareh Eskandari]{Setareh Eskandari}
\address{Setareh Eskandari  \\Department of Mathematics and Mathematical Statistics \\
	Ume{\aa} University\\
	90187 Ume{\aa}\\
	SWEDEN.}\email{setareh.eskandari@umu.se}

\author[Antti Per\"{a}l\"{a}]{Antti Per\"{a}l\"{a}}
\address{Antti Per\"{a}l\"{a} \\Department of Mathematics and Mathematical Statistics \\
	Ume{\aa} University\\
	90187 Ume{\aa}\\
	SWEDEN.} \email{antti.perala@umu.se}



%


\thanks{S. Eskandari was funded by the postdoctoral scholarship JCK22-0052 from the Kempe Foundations. We would like to thank the anonymous referee for the valuable comments that have improved the paper.}


\begin{abstract}
We study the optimal domains for bounded Volterra integration operators $T_g$ between Hardy spaces $H^p$ and $H^q$ of the unit ball. It is shown that the optimal domain of a bounded $T_g:H^p\to H^q$ always strictly contains $H^p$. Moreover, the intersection of the optimal domains is equal to $H^p$ if $p\geq q$, whereas if $p<q$, we show that this intersection is a genuinely larger tent space of holomorphic functions. In the unit disk, this problem was recently solved for $p=q$ by Bellavita, Daskalogiannis, Nikolaidis and Stylogiannis.
\end{abstract}

\maketitle



\section{Introduction}

\noindent Given $z=(z_1,z_2,...,z_n)$ and $\zeta=(\zeta_1,\zeta_2,...,\zeta_n)$, two vectors of the complex space $\mathbb{C}^n$, we denote $$\langle z,\zeta\rangle=\sum_{k=1}^n z_k\overline{\zeta_k}.$$ We set $|z|=\sqrt{\langle z,z\rangle}$ and let $\mathbb{B}_n=\{z\in \mathbb{C}^n:|z|<1\}$ be the open unit ball, and $\mathbb{S}_n=\partial \mathbb{B}_n=\{z\in\mathbb{C}^n: |z|=1\}$ its boundary, the unit sphere. Let $H(\mathbb{B}_n)$ denote the space of all holomorphic functions on $\mathbb{B}_n$, equipped with the topology of uniform convergence on compact sets. We use notation $R$ for the radial derivative: 

$$Rf(z)=\sum_{k=1}^n z_k\partial_{z_k}f(z).$$

Given $g \in H(\mathbb{B}_n)$ the operators $T_g$ and $S_g$ acting on $H(\mathbb{B}_n)$ given by
$$T_g f(z)=\int_0^1 f(tz)Rg(tz)\frac{dt}{t},\quad S_g f(z)=\int_0^1 Rf(tz)g(tz)\frac{dt}{t}$$
are well-defined. The operator $T_g$ is called the (generalized) Volterra integration operator, while $S_g$ is called the Volterra companion operator. Pommerenke \cite{Pom} was the first to study $T_g$ in one dimension; the variant we use for the unit ball appears first in the work of Hu \cite{Hu}.

Assume now that $X$ and $Y$ are two topological vector spaces of holomorphic functions on $\mathbb{B}_n$. Inspired by \cite{SZ}, we consider
$$V_{X,Y}=\{g \in H(\mathbb{B}_n): T_g \text{ is continuous } X\to Y\}.$$
Clearly $V_{X,Y}$ is always a vector space. If $X$ and $Y$ are quasi-Banach spaces, then $V_{X,Y}$ is a quasinormed space equipped with the topology coming from $\|g\|_{V_{X,Y}}=\|T_g\|_{X\to Y}+|g(0)|$. Moreover, if convergence in both $X$ and $Y$ implies uniform convergence on compact sets (that is, $X$ and $Y$ continuously embed into $H(\mathbb{B}_n)$), and they both contain the constants, then $V_{X,Y}$ is also complete. This fact is proven in \cite{SZ} in the setting of one complex variable when $T_g$ acts $X\to X$, with a slightly less restrictive requirements on $X$. The reader should have no difficulties verifying that it also holds for $T_g:X\to Y$ under the assumptions above. 

The classical problem for $T_g$ is to study its boundedness between two Hardy spaces. We recall that for $0<p<\infty$, the Hardy space $H^p$ consists of $f \in H(\mathbb{B}_n)$ with
$$\|f\|_{H^p}=\sup_{0\leq s <1}\left(\int_{\mathbb{S}_n}|f(s \zeta)|^pd\sigma(\zeta)\right)^{1/p}<\infty,$$
where $d\sigma$ is the normalized surface measure on $\mathbb{S}_n$. When $p\geq 1$, the norm $\|\cdot\|_{H^p}$ makes $H^p$ a Banach space. When $0<p<1$, the space $H^p$ is a complete quasi-Banach space. Moreover, in that case $d_p(f,g)=\|f-g\|_{H^p}^p$ is a complete translation invariant metric on $H^p$, making it an F-space in the sense it is used by Rudin \cite{Rud} (F-spaces need not be locally convex). For basic properties of these spaces, we refer the reader to \cite{Dur, ZhuHol, Zhu}.

For $T_g:H^p\to H^q$, we write $V_{p,q}=V_{H^p,H^q}$. These spaces are completely known. In fact (see  and \cite{Pau}, or \cite{AC, AS} for earlier proofs in the case $n=1$),

\begin{itemize}
\item $p=q$:\quad $V_{p,q}=BMOA$;
\item $p<q$:\quad $V_{p,q}=\mathcal{B}^\alpha$;
\item $p>q$:\quad $V_{p,q}=H^r$.
\end{itemize}
Here $r=pq/(p-q)$, $\alpha=1+n/q-n/p$ and $\mathcal{B}^\alpha$ is the Bloch type space consisting of analytic functions $g$ with
$$\|g\|_{\mathcal{B}^\alpha}=\sup_{z\in\mathbb{B}_n}(1-|z|^2)^\alpha |Rg(z)|<\infty,$$ and we understand this space to consist of constants only, if $\alpha<0$ (in which case $T_g$ is the zero operator). The space $BMOA$ is the space of holomorphic functions of bounded mean oscillation -- the precise definition is not used in this paper, but it can be found in \cite{ZhuHol} and in \cite{Dur, Zhu} for $n=1$.

Given $X$, $Y$, and $g \in V_{X,Y}$, we may ask which holomorphic functions $f$ are mapped into $Y$ by $T_g$. This vector space of holomorphic functions will be called the optimal domain of $T_g$ and denoted
$$[g:Y]=\{f \in H: T_g f\in Y\}.$$
This space $[g:Y]$ can be equipped with the quasinorm $\|f\|_{[g:Y]}=\|T_gf\|_Y$. The study of optimal domains for integration operators was initiated by Curbera and Ricker \cite{CR1, CR2}, who looked into this concept for the Ces\'aro operator. In the recent years, there has been a number of interesting papers on the topic, see \cite{ABR1, ABR2, BDNS}, or \cite{BBNS, LS} for meromorphic optimal domains in the context of Hardy and Bergman-Morrey spaces. We mention that all the aforementioned papers are written in the setting of the unit disk. We also remark that our notation differs slightly from them. However, in the context of Volterra operators only, this notation should not confuse the reader. We have chosen it because it makes the notation $[V_{X,Y}:Y]$ that is given below more convenient.

From the definitions it follows that we always have $X\subseteq [g:Y]$ for $g \in V_{X,Y}$, and if, for instance $X=Y=H^p$ and $n=1$, the inclusion is known to be proper \cite{BDNS}. However, another interesting question is whether the intersection of all the optimal domains agrees with $X$. That is, whether $[V_{X,Y}:Y]=X$, where
$$[V_{X,Y}:Y]=\bigcap_{g \in V_{X,Y}} [g:Y].$$
We understand this problem in the following way. While $X$ and $Y$ uniquely define $V_{X,Y}$, when is it true that $Y$ and $V_{X,Y}$ uniquely define $X$? Or, in other words, can $X$ be recovered from the optimal domains of bounded Volterra operators $X\to Y$.

Before going further, let us offer a simple example. Let $X=A(\mathbb{B_n})$ be the ball algebra and $Y=\mathcal{B}^1$ be the standard Bloch space. It is easy to see that $V_{X,Y}=\mathcal{B}^1$, and that for every $g \in V_{X,Y}$ one has $H^\infty\subseteq [g:\mathcal{B}^1]$, so clearly $X \subsetneq [V_{X,Y}:Y]$ in that case. But one might argue that an example such as this one is tailor-made for the identity to fail. So what happens in the Hardy spaces? 

We will prove that the spaces $[g:H^q]$ are always complete in their respective topologies, and that if $T_g:H^p\to H^q$ is bounded, then $H^p$ is strictly contained in $[g:H^q]$. In \cite{BDNS} it is proven that
\begin{equation}\label{orig}
[BMOA:H^p]=H^p,\quad p\geq 1,
\end{equation}
when $n=1$.
In this work, we work in the setting of the unit ball and prove that if $p>q$ and $r=pq/(p-q)$, then
$$[H^r:H^q]=H^p,$$
which means that $H^p$ can be recovered from the optimal domains of bounded $T_g:H^p\to H^q$ for $p>q$. On the other hand, when $p<q$ and $\alpha=1+n/q-n/p$, we prove
$$H^p \subsetneq [\mathcal{B}^\alpha:H^q],$$
and moreover $[\mathcal{B}^\alpha:H^q]$ has a simple description as a tent space. We also establish the \eqref{orig} in for unit ball.

The results indicate, in particular, that $H^p$ cannot be recovered from optimal domains of bounded $T_g:H^p\to H^q$ for $p<q$. Other properties of optimal domains of this type were very recently considered in \cite{ABR2} for $n=1$.

In the next section, we focus on the case $p>q$, but we also prove some basic results on optimal domains such as the fact that $[g:H^q]$ is always a quasi-Banach space (Theorem \ref{complete}). For the full characterization of the intersection of optimal domains in the case $p>q$, which is Theorem \ref{compact}, we will also characterize the boundedness of $S_g$ between two Hardy spaces. This is key to studying optimal domains using the approach in \cite{BDNS}. The third section is devoted to the study of the case $p<q$, most notably Theorem \ref{tentdom} where we also characterize the intersection of the optimal domains, providing the necessary tent space definitions.

The fourth section deals with the case $p=q$, where the problem was originally studied. We verify that $H^p$ is always strictly contained in the corresponding optimal domain, and record as a consquence of this and the previous results that is always the case for a bounded $T_g:H^p\to H^q$ in Corollary \ref{always}. Finally, we show that $H^p$ can be recovered for the optimal domains of $T_g$ for every $p>0$, see Theorem \ref{same}. Perhaps surprisingly, this is the most difficult of the three main proofs, and that is why it is saved for the last. This proof, as well as the other two main proofs, makes use of the theory on tent spaces, together with Kahane and Khinchine inequalities and atomic decompositions.

As one might expect, we will need comparisons between different quantities. As is standard, by $A\lesssim B$ (or equivalently $B\gtrsim A$), we mean that there exists $C>0$ so that $A\leq CB$. If both $A \lesssim B$ and $B\gtrsim A$ are true, we write $A\asymp B$.

\section{The case $p>q$}

\noindent We will begin by proving that the optimal domains $[g:H^q]$ for a non-constant holomorphic $g$ are always quasi-Banach spaces, and Banach spaces, if $q\geq 1$. If $g$ is constant, then the optimal domain is trivially all of $H(\mathbb{B}_n)$, and the construction gives it the trivial topology, making this case meaningless for further analysis.

\begin{theorem}\label{complete}
Let $0<q<\infty$ and $g$ be holomorphic and non-constant. Then the optimal domain $[g:H^q]$ is a quasi-Banach space. It is further a Banach space if $q\geq 1$.
\end{theorem}

\begin{proof}
The map $f\mapsto T_gf$ is linear. Therefore it is clear that the topology in $[g:H^q]$ is given by a norm or a quasi-norm, whenever the topology of $H^q$ is. We will investigate the completeness of $[g:H^q]$. The argument is largely the same as in Theorem 1 of \cite{BDNS}.

Let $f \in [g:H^q]$. Let $Z_{Rg}=\{z\in \mathbb{B}_n: Rg(z)=0\}$. For each point $z \notin Z_{Rg}$, we have 
$$|f(z)|\leq C_z \|f\|_{[g:H^q]},$$
which follows for using the boundedness of the functional $h\mapsto Rh(z)$ on $H^q$. The positive constants $C_z$ depend on $1/|Rg(z)|$ and $q$ but vary continuously on $\mathbb{B}_n\setminus Z_{Rg}$.

Let now $z \in Z_{Rg}$. This case needs some extra effort, since the zeroes of holomorphic functions of several complex variables need not be isolated inside the domain. We proceed as follows. Since $g$ is not constant, there exists $\zeta \in \mathbb{C}^n$ with $|\zeta|=1$, and $\varepsilon>0$, so that the one variable analytic slice $h_\zeta(w)=Rg(z+w\zeta)$ satisfies $|h_\zeta(w)|>\varepsilon$ when $r<|w|<R$ for some $0<r<R$. We can use subharmonicity of $\log|h_\zeta|$ to arrive at the an upper bound for $|f(z)|$ similar to that in \cite{BDNS}. This upper bound can be chosen to be uniform in some neighborhood of $z$ (regardless if the new point belong to $Z_{Rg}$ or not).

For these observations, it follows that Cauchy sequences in $[g:H^q]$ converge to a holomorphic function uniformly on compact subsets. Now that this has been achieved, the rest of the argument in the cited paper carries over.
\end{proof}

Let $ 0<q<p<\infty $ and $ r= pq/(p-q)$. We show that if $T_g:H^p\to H^q$ is bounded, then $H^p$ is always strictly contained in the optimal domain $[g:H^q]$.

\begin{proposition}\label{pbig}
Let $ 0<q<p<\infty $, $g \in H^r$. Then $H^p\subsetneq [g:H^q]$.
\end{proposition}

\begin{proof}
One can argue as in \cite{BDNS}. Indeed, for any $g \in H^r$, it is trivially true that $H^p$ is contained in $[g:H^q]$. If there were a $g \in H^r$ with $H^p=[g:H^q]$, the identity map would be bounded and onto $H^p\to [g:H^q]$. The open mapping theorem (which holds also for F-spaces \cite{Rud}) then implies that $T_g:H^p\to H^q$ is bounded below.

Since $H^p$ is a subspace of $[g:H^q]$ and embeds there boundedly, if there was a $g \in H^r$ with $H^p=[g:H^q]$, one would have that $T_g:H^p\to H^q$ is bounded below. It is well-known \cite{Pau} that for this range of $p$ and $q$, $T_g:H^p\to H^q$ is bounded if and only if it is compact. A compact operator cannot be bounded below.
\end{proof}

We will now proceed with the proofs of the results mentioned in the Introduction.

\begin{theorem}\label{compact}
	Let $ 0<q<p<\infty $ and $ r= pq/(p-q)$. Then $ H^p = [H^r:H^q]$.
\end{theorem}
\begin{proof}
We follow the reasoning in \cite{BDNS}. Let $ g\in V_{p,q} = H^r$, and $f\in  [H^r: H^q]$. Now, $  T_g f $ belongs to $H^q$, or equivalently, $S_f  g \in H^q$.  This means by the closed graph theorem (for F-spaces if necessary; see \cite{Rud}) that $S_f:H^r\to H^q$ is bounded. The claim now follows from the next theorem.
\end{proof}	

We will formulate the following theorem in the standard form. That is, for $S_g:H^p\to H^q$. Note that if $p>q$, the roles of $p$ and $r$ are symmetric, so we will be able to finish the proof once this result is obtained.
	
\begin{theorem}\label{p>q}
Let $ 0<q<p<\infty $, $ r= pq/(p-q)$ and $g \in H(\mathbb{B}_n)$. Then $S_g: H^p \to H^q$  is bounded if and only if $ g\in H^r $. Moreover, if $g(0)=0$, we have
$$\|S_g\|_{H^p\to H^q}\asymp \|g\|_{H^r}.$$
\end{theorem}	
\begin{proof}

We note that this result is probably known to experts, at least when $n=1$. Since we have been unable to find its proof in the literature, one is given here.

We first consider the sufficiency. In fact, it follows easily from the corresponding results for $T_g$ and the multiplication operator $M_g$. We provide a direct proof, similar to one for $T_g$ in \cite{Pau}, by methods that will be used for the necessity. For $\gamma>1$ and $|\zeta|=1$, we denote by $\Gamma(\zeta)=\Gamma_{\gamma}(\zeta)$ the approach region
$$\Gamma_{\gamma}(\zeta)=\left\lbrace z \in \mathbb{B}_n:|1-\langle z,\zeta\rangle|<\frac{\gamma}{2}(1-|z|^2)\right\rbrace.$$

For $\beta \in \mathbb{R}$ we write $dV_\beta(z)=(1-|z|^2)^\beta dV(z)$, where $V$ is the normalized $2n$-dimensional Lebesgue measure on $\mathbb{B}_n$. Now assume $g\in H^r $ with $ r= pq/(p-q)$. If $f\in H^p  $, then by the area function description of Hardy spaces \cite{AB, Pau} we have
\begin{align*}
	\|S_g f\|^q_{H^q} &\asymp \int_{\mathbb{S}_n} \left(\int_{\Gamma(\zeta)} |Rf(z) g(z) |^2 dV_{1-n}(z) \right)^{q/2} d\sigma(\zeta)\\
	&\lesssim \int_{\mathbb{S}_n}\left( \sup_{z\in\Gamma(\zeta)} |g(z)|\right)^q \left(\int_{\Gamma(\zeta)} |Rf(z)|^2 dV_{1-n}(z)\right)^{q/2}d\sigma(\zeta).
\end{align*}
We use Hölder's inequality and get
\begin{align*}
	\|S_g f\|^q_{H^q} \lesssim &\left( \int_{\mathbb{S}_n}\left( \sup_{z\in\Gamma(\zeta)} |g(z)|\right)^{pq/(p-q)} d\sigma(\zeta)\right)^{1-q/p} \\&\times\left(\int_{\mathbb{S}_n}\left(\int_{\Gamma(\zeta)} |Rf(z)|^2 dV_{1-n}(z)\right)^{p/2}d\sigma(\zeta)\right)^{q/p}.
\end{align*}
Equivalently, this means
$$\|S_g f\|^q_{H^q} \lesssim \|g\|^q_{H^r} \|f\|^q_{H^p}, $$ 

giving the sufficiency.\\

For the necessity, we proceed with the argument from \cite{MPPW}, based on the Kahane and Khinchine inequalities \cite{Lue}, atomic decomposition \cite{CR, Pau}, area function description of Hardy spaces \cite{AB, Pau}, and factorization of tent spaces \cite{CV, CMS, MPPW}. The proof follows the outline of either Theorem 7 or Theorem 8 of \cite{MPPW}, except that it is simpler. The proof is sketchy in the sense that we do not discuss the preliminaries on tent spaces is this paper; an interested reader can take a look at \cite{MPPW}.

Let $Z=\{a_k\}\subseteq \mathbb{B}_n$ be a $\rho$-lattice in the Bergman metric. The most of the argument below works for any $\rho>0$, but in the very last stage we will assume that $\rho$ is very small. Let $b>n\max(1,2/p)$. We consider a family of test functions of the form 
$$ f_t(z) = \sum_k \lambda_k r_k(t) \left(\frac{1-|a_k|^2}{1-\langle z,a_k\rangle} \right)^b,$$

where the sequence $ \{\lambda_k\} $ belongs to $ T_2^p$, a tent space of sequences indexed by the lattice, and $r_k : [0,1]\to \{-1 , +1\}  $ are the Rademacher functions. By the area function description of the Hardy spaces, integrating with respect to t, where $ t\in [0,1], $ and using Fubini's theorem, we have 
\begin{align*}
	\int_{\mathbb{S}_n} \int_{0}^{1} &\left( \int_{\Gamma(\zeta)} \left|g(z) \sum_{k} \lambda_k r_k(t) |\langle z,a_k\rangle|\dfrac{(1-|a_k|^2)^b}{(1-\langle a_k,z\rangle)^{b+1}} \right|^2 dV_{1-n}(z)\right)^{q/2} dt \,d\sigma(\zeta)\\
	&\lesssim \|S_g\|^q \,\|\lambda\|^q_{T^p_2}.
\end{align*}
Now, using Kahane's inequality and Fubini's theorem this gives
\begin{align*}
	\int_{\mathbb{S}_n} &\left(  \int_{\Gamma(\zeta)}\int_{0}^{1} \left|g(z)  \langle z,a_k\rangle \sum_{k} \lambda_k r_k(t) \dfrac{(1-|a_k|^2)^b}{(1-\langle z,a_k\rangle)^{b+1}} \right|^2 dt \,dV_{1-n}(z)\right)^{q/2}  \,d\sigma(\zeta)\\
	&\lesssim \|S_g\|^q \,\|\lambda\|^q_{T^p_2}.
\end{align*}
We use Khinchine's inequality to obtain
\begin{align*}
	\int_{\mathbb{S}_n} &\left(  \int_{\Gamma(\zeta)} \sum_{k}   |\lambda_k|^2 |\langle z,a_k\rangle|^2 |g(z)|^2  \dfrac{(1-|a_k|^2)^{2b}}{|1-\langle z,a_k\rangle|^{2b+2}}   \,dV_{1-n}(z)\right)^{q/2}  \,d\sigma(\zeta)\\
	&\lesssim \|S_g\|^q \,\|\lambda\|^q_{T^p_2}.
\end{align*}
Replacing the integral over $ \Gamma(\zeta) $ by an integral over Bergman metric balls $B(a_k, \rho)$, and switching to an approach region with a smaller aperture, also denoted by $\Gamma(\zeta)$,
\begin{align*}
	\int_{\mathbb{S}_n} &\left( \sum_{{a_k}\in \Gamma(\zeta)}    |\lambda_k|^2 \int_{B(a_k, r)}|\langle z,a_k\rangle |^2 |g(z)|^2  \dfrac{(1-|a_k|^2)^{2b}}{|1-\langle z,a_k\rangle|^{2b+2}}   \,dV_{1-n}(z)\right)^{q/2}  \,d\sigma(\zeta)\\
	&\lesssim \|S_g\|^q \,\|\lambda\|^q_{T^p_2}.
\end{align*}
We proceed by using subharmonicity and $|1-\langle z,a_k\rangle| = 1-|a_k|^2 $ for every $ z\in B(a_k, r)$, arriving at the estimate 
\begin{equation}\label{equ1}
	\int_{\mathbb{S}_n} \left( \sum_{{a_k}\in \Gamma(\zeta)}    |\lambda_k|^2 |a_k|^4 |g(a_k)|^2  \right)^{q/2}  \,d\sigma(\zeta) \lesssim \|S_g\|^q \,\|\lambda\|^q_{T^p_2}.
\end{equation}
Denoting $\nu_k = |a_k|^2 |g(a_k)|,$  we want to prove that $ \{\nu_k\} \in T^r_\infty$. This would lead to a discrete way to express that $g\in H^r$.    

Let $s>0$ be big enough so that $2s>1$ and $qs>1$ hold. We are done if we can prove that $\{\nu_k^{1/s}\}\in T^{rs}_\infty$. By duality of tent spaces, since $rs>1$, this can be shown if we prove
$$\left|\sum_k \mu_k \nu_k^{1/s}(1-|a_k|^2)^n\right|\lesssim \|\mu\|_{T^{(rs)'}_1}.$$

To this end, we assume that $\lambda_k$ and $\tau_k$ are positive. Indeed, since $\nu_k$ are positive, the above dual pairing is clearly at its biggest for a positive sequence $\{\mu_k\}$, and the sign change on individual elements does not affect the norm of the corresponding tent space. We get
$$\sum_k \tau_k \lambda_k^{1/s} \nu_k^{1/s}  ( 1- |a_k|^2)^n \asymp  	\int_{\mathbb{S}_n} \left( \sum_{{a_k}\in \Gamma(\zeta)}    \tau_k \lambda_k^{1/s} \nu_k^{1/s}  \right)  \,d\sigma(\zeta). $$
Next, using Hölder's inequality, yields
\begin{align*} &\sum_k \tau_k \lambda_k^{1/s} \nu_k^{1/s}  ( 1- |a_k|^2)^n \\
\lesssim & 	\int_{\mathbb{S}_n} \left( \sum_{{a_k}\in \Gamma(\zeta)}    \tau_k^{2s/(2s-1)} \right)^{(2s-1)/2s} \left(\sum_{{a_k}\in \Gamma(\zeta)}  \lambda_k^2 \nu_k^2 \right)^{1/(2s)} \,d\sigma(\zeta).
\end{align*}
Consider $ 2s/(2s-1) = (2s)' $ and doing another Hölder's inequality for $qs$ and $qs/(qs-1) = (qs)'$ gives us now that the expression is dominated by

$$ \left(	\int_{\mathbb{S}_n} \left( \sum_{{a_k}\in \Gamma(\zeta)}    \tau_k^{\frac{2s}{2s-1}} \right)^{\frac{(qs)'}{(2s)'}}\,d\sigma(\zeta)\right)^{\frac{1}{(qs)'}}
\left( \int_{\mathbb{S}_n} \left(\sum_{{a_k}\in \Gamma(\zeta)}  \lambda_k^2 \nu_k^2 \right)^{\frac{q}{2}} \,d\sigma(\zeta)\right)^{\frac{1}{qs}}.$$

The first part is the norm $\|\tau\|_{T^{(qs)'}_{(2s)'}}.$ By using \ref{equ1} for the second part, to get the upper bound $\|S_g\|^{1/s} \|\lambda\|^{1/s}_{T^p_2} = \|S_g\|^{1/s} \|\lambda\|_{T^{ps}_{2s}}$. Altogether we have
$$ \sum_k \tau_k \lambda_k^{1/s} \nu_k^{1/s}  ( 1- |a_k|^2)^n  \lesssim \|S_g\|^{1/s} \|\lambda^{1/s}\|_{T^{ps}_{2s}} \|\tau\|_{T^{(qs)'}_{(2s)'}}.$$

By factorization of tent spaces (essentially due to Cohn and Verbitsky \cite{CV}, but the present variant is taken from \cite{MPPW}) and considering $\mu_k =\tau_k \lambda_k^{1/s}$ as an arbitrary positive element of the space $T^{(rs)'}_1,$ using duality we have $\{\nu_k^{1/s}\} \in T^{rs}_\infty$. By the discussion and calculation above, this leads to $\|\nu\|_{T^r_\infty} \lesssim \|S_g\|,$ which is what we wanted to prove. 

Now, choosing $\rho>0$ small enough and invoking Lemma 3 of \cite{MPPW}, we get
$$\int_{\mathbb{S}_n}\left(\sup_{z\in \Gamma(\zeta)}|z_i|^2 |g(z)|\right)^rd\sigma(\zeta)<\infty$$

for every $i\in \{1,2,...,n\}$ (since $z_i^2g(z)$ is holomorphic), yielding that 

\begin{equation}\label{Tr}\int_{\mathbb{S}_n}\left(\sup_{z\in \Gamma(\zeta)}|z|^2 |g(z)|\right)^rd\sigma(\zeta)<\infty.
\end{equation}

By the admissible maximal function description of Hardy spaces, we see that $H^r$ and the tent space of holomorphic functions defined by \eqref{Tr} are in bijective correspondence. By the Bounded Inverse Theorem for F-spaces, we get that $g \in H^r$ with its norm dominated by the operator norm of $S_g$.

\end{proof}

\section{The case $p<q$}

\noindent If $0<p<q<\infty$ and $T_g:H^p\to H^q$ is compact, then Proposition 3.19 of \cite{ABR2} states that $H^p\subsetneq [g:H^q]$. In fact, the result in the cited paper is formulated on the unit disk, but extends to several complex variables without extra efforts. In this section we will show that the same is true for the intersection of the optimal domains, even when taken over all bounded $T_g$. Morever, we completely describe this intersection.

As an appetiser, we will present a short proof that the intersection of the optimal domains is always larger than $H^p$, when $p<q$. We will need to introduce the standard weighted Bergman spaces. Let $\beta>-1$ and $0<p<\infty$. The weighted Bergman space $A^p_\beta$ consists of $f \in H(\mathbb{B}_n)$ with
$$\|f\|_{A^p_\beta}=\left(\int_{\mathbb{B}_n}|f(z)|^p dV_\beta(z)\right)^{1/p}<\infty.$$
See \cite{ZhuHol, Zhu} for some basic theory of $A^p_\beta$. We will need the following inclusion between Hardy and Bergman spaces for $p<s$.

\begin{equation}\label{HpAs}
H^p\subsetneq A^s_{n\frac{s}{p}-n-1}.
\end{equation}

In the unit disk this result goes back at least to 1932 paper of Hardy and Littlewood \cite{HL}. It is also a special case of Duren's theorem \cite{Dur1}. It is moreover known that the spaces $A^s_{n\frac{s}{p}-n-1}$ increase in size as $s$ increases. A simple proof of \eqref{HpAs} can be found in \cite{ZhuHol} (see Theorem 4.48).

\begin{proposition}\label{berg}
Let $ 0<p<q<\infty $ and $\alpha=1+n/q-n/p$. Then one has $ H^p \subsetneq [\mathcal{B}^\alpha: H^q]$.
\end{proposition}

\begin{proof}
Assume $p<q$ and $g\in \mathcal{B}^{\alpha}$. We consider two cases, $q>2$ and $q\leq2$. For the case $q>2$, let $s\in (2,q)\cap(p,q)$. Then we have $2<s<q$ and from Theorem 1 (1) of \cite{MPPW}, it follows that $T_g: A^s_{n\frac{s}{p}-n-1} \to H^q$ is bounded. Therefore $ A^s_{n\frac{s}{p}-n-1} \subseteq [g : H^q]$.

Let us next assume that $q\leq 2 $. Here we could also proceed by taking $s\in (p,q)$, in which case $s\leq \min\{2,q\}$. But in fact we can let $s=q$, obtaining a larger space. According to case (1) of Theorem 1 in \cite{MPPW} (in fact, in the unit disk this case was obtained earlier by Wu \cite{Wu}), the boundedness of $T_g : A^q_{n\frac{q}{p}-n-1} \to H^q$ is equivalent to $ g\in \mathcal{B}^{\alpha}$.

Now, by using \eqref{HpAs}, the space $H^p$ embeds into $A^s_{n\frac{s}{p}-n-1}$ or $A^q_{n\frac{q}{p}-n-1}$, whichever we want. The inclusion is also strict as Bergman functions need not have boundary values almost everywhere. Therefore the proof is complete.  
\end{proof}

The following result clearly follows from the above theorem. We record it here.

\begin{coro}\label{qbig}
If $0<p<q<\infty$ and $T_g:H^p\to H^q$ is bounded, then $H^p\subsetneq [g:H^q]$.
\end{coro}

It turns out that, unless $q=2$, even the Bergman spaces above are smaller than the intersection. In fact, the space $[\mathcal{B}^\alpha:H^q]$ has a convenient description as a tent space.

Let $\omega\geq 0$ be a measurable function. We denote by $AT^p_q(\omega)$ the weighted tent space (originally introduced by Coifman, Meyer and Stein \cite{CMS}) consisting of analytic functions $f$ with
$$\|f\|_{AT^p_q(w)}=\left(\int_{\mathbb{S}_n}\left(\int_{\Gamma(\zeta)}|f(z)|^q w(z)dV_{1-n}(z)\right)^{p/q}d\sigma(\zeta)\right)^{1/p}.$$
These spaces are very useful for a great number of problems on the unit ball; see \cite{MPPW, OF, Pau, Per}.

When we have a standard weight $\omega(z)=(1-|z|^2)^\beta$, we write $AT^p_q(\omega)=AT^p_{q,\beta}$. These spaces will be F-spaces, at least when $q\geq 1$. We note that this space is non-trivial if $\beta>-2$, due to how we have decided to parametrise the weight range due to it being convenient for the area description of Hardy spaces. We can now precisely describe the intersection of the optimal domains in the case $p<q$.

\begin{theorem}\label{tentdom}
Let $0<p<q<\infty$ and $\alpha=1+n/q-n/p\geq 0$. We then have
$$AT^q_{2,-2\alpha}=[\mathcal{B}^\alpha:H^q].$$
If $\alpha<0$, then $[\mathcal{B}^\alpha:H^q]$ consist of all holomorphic functions on the ball.
\end{theorem}

\begin{proof}
We first prove that $AT^p_{2,-2\alpha}$ is contained in $[\mathcal{B}^\alpha:H^q]$. Using the area description of $H^q$, one obtains
$$\|T_gf\|_{H^q}^q\asymp\int_{\mathbb{S}_n}\left(\int_{\Gamma(\zeta)}|f(z)|^2|Rg(z)|^2 dV_{1-n}(z)\right)^{q/2}d\sigma(\zeta).$$
Bearing in mind $g \in \mathcal{B}^\alpha$, we use the pointwise estimate
$$|Rg(z)|\leq (1-|z|^2)^{-\alpha}\|g\|_{\mathcal{B}^\alpha}.$$
We get
$$\|T_gf\|_{H^q}^q \lesssim \|g\|_{\mathcal{B}^\alpha}^q\int_{\mathbb{S}_n}\left(\int_{\Gamma(\zeta)}|f(z)|^2 (1-|z|^2)^{-2\alpha}dV_{1-n}(z)\right)^{q/2}d\sigma(\zeta).$$
This gives exactly that $T_g:AT^q_{2,-2\alpha}\to H^q$ is bounded. Observe that $\alpha<1$, so the space in non-trivial. So $AT^p_{2,-2\alpha}$ is contained in $[\mathcal{B}^\alpha:H^q]$.

We prove the other inclusion. It suffices to show that $S_g:\mathcal{B}^\alpha\to H^q$ being bounded implies that $g \in AT^q_{2,-2\alpha}$. So let us assume that $S_g$ is bounded.

We first deal with the case $\alpha=0$. In that case, we test with $F_k(z)=z_k$, and use the area description of of $H^q$ as above. This shows that $z_k g(z)$ belongs to $AT^q_{2,0}$ for every $k=1,2,...,n$. As in the end of the proof of Theorem \ref{p>q}, the leads to $g \in AT^q_{2,0}$.

We now turn to the case $0<\alpha<1$. The idea is similar to that of Theorem \ref{p>q}, and to this end we use the atomic decomposition of $\mathcal{B}^\alpha$. This result can be found in \cite{ZZ} (if $0<\alpha<1$, the space $\mathcal{B}^\alpha$ is the same as the Lipschitz space $\Lambda_{1-\alpha}$ in the notation and terminology of the reference), where it is explained how it directly follows from the atomic decomposition of the Bloch space \cite{ZhuHol}. Let $\{a_k\}$ be a $\rho$-lattice in the Bergman metric. For every $\{\lambda_k\}\in \ell^\infty$, the functions
$$F_t(z)=\sum_{k}\lambda_k r_k(t)\frac{(1-|a_k|^2)^b}{(1-\langle z,a_k\rangle)^{b+\alpha-1}}$$
belong to $\mathcal{B}^\alpha$, when $b>n$. The functions $r_k$ are the Rademacher functions like in Theorem \ref{p>q}.

By using the set of of techniques demonstrated in the proof of Theorem \ref{p>q}, and the calculations leading up to \eqref{equ1} to control $S_gF_t$, one arrives at the estimate
\begin{equation}\label{alpha}
\int_{\mathbb{S}_n}\left(\sum_{a_k \in \Gamma(\zeta)} |\lambda_k|^2 |a_k|^4 |g(a_k)|^{2} (1-|a_k|^2)^{-2\alpha+2}\right)^{q/2}d\sigma(\zeta).
\end{equation}
If we select $\lambda_k=1$ for every $k$, we get the discrete variant of what we want. Assuming $\rho>0$ is small enough, we can use Lemma 3 of \cite{MPPW} and reason as in the proof of Theorem \ref{p>q} to get that $ g \in AT^q_{2,-2\alpha}$.

The latter claim is clear: If $\alpha<0$, then $\mathcal{B}^\alpha$ consists of the constant functions only, and $T_g$ is the zero operator.
\end{proof}

We note that $$AT^2_{2,-2\alpha}=A^2_{2,-2-n+\frac{2n}{p}}=A^2_{\frac{2n}{p}-n-1},$$ so the simple argument in Proposition \ref{berg} is enough to find the correct intersection in this case.

\section{The case $p=q$ in the unit ball}

\noindent Since the case $p=q$ is where the problem was initially studied for the unit disk, investigation of this case is in order for higher dimension. We first observe that if $g \in BMOA$, then $H^p\subsetneq [g:H^p]$. The proof follows the same reasoning as the one-dimensional proof in \cite{BDNS}, where it can be reduced to the case $p=q=2$ and a theorem due to Anderson \cite{And}.

When $n=1$, one can use monomials $z^k$ as test functions. In fact, these functions being suitable is closely related to them being inner functions on the unit disk. When $n\geq 2$, one needs to employ inner functions on the unit ball. The existence of non-constant inner functions in higher dimension is by no means a trivial matter. The first construction was in 1982 and due to Aleksandrov \cite{Alek}, but around the same time, the works of Hakim and Sibony \cite{HS} and L\o w \cite{Loew} obtained the same result independently. A reader interested in the development of this interesting topic might want to look at the following paper of Rudin \cite{RudIn}.

\begin{theorem}\label{pisq}
Let $0<p<\infty$ and $g \in BMOA$. Then $H^p\subsetneq [g:H^p]$.
\end{theorem}

\begin{proof}
By the same reasoning as in Theorem 2 of \cite{BDNS} -- see also the proof of Theorem \ref{pbig} in the present paper -- it suffices to show that $T_g:H^p\to H^p$ is not bounded below.

We first consider the case $p=2$. If $f \in H^2$, then the $H^2$ norm of $T_gf$ is comparable to
\begin{equation}\label{H2}
\int_{\mathbb{B}_n}|f(z)|^2 |Rg(z)|^2 (1-|z|^2)dV(z).
\end{equation}
Note that $g \in BMOA$ if and only if the measure $d\mu_g(z)=|Rg(z)|(1-|z|^2)dV(z)$ is a Carleson measure, so that $H^2$ embeds boundedly into $L^2_{\mu_g}$. Now, let $\varphi$ be an inner function on the unit ball, and consider $f_k=\varphi^k$. With these choices, one has $|\varphi(z)|< 1$ for every point $z \in \mathbb{B}_n$ and $\|f_k\|_{H^2}=1$ for every $k$. By the dominated convergence theorem, the equivalent norms \eqref{H2} tend to zero as $k\to \infty$. This shows that $T_g:H^2\to H^2$ is not bounded below.

Now, if $p\neq 2$, the argument of Lemma 1 of \cite{BDNS} shows how to reduce this case to the case $p=2$ above.
\end{proof}

We can collect Theorem \ref{pbig}, Corollary \ref{qbig}, and Theorem \ref{pisq} as one corollary.

\begin{coro}\label{always}
Let $0<p,q<\infty$ and $T_g:H^p\to H^q$ be bounded. Then we have $H^p\subsetneq [g:H^q]$.
\end{coro}

As mentioned in the beginning of the paper, when $n=1$ the identity $H^p=[BMOA:H^p]$ for $p\geq 1$ was proved in \cite{BDNS}. We extend this result to the Hardy spaces of the unit ball for the full range $0<p<\infty$. We note that the proof in completely different from \cite{BDNS}, which is based on recovering $g$ from $S_g$ by feeding the operator suitable simple test functions. We have not been able the carry out the same proof in the unit ball.

\begin{theorem}\label{same}
Let $0<p<\infty$. We have $H^p=[BMOA:H^p]$.
\end{theorem}

\begin{proof}
We only need to show that each element in $[BMOA:H^p]$ is an $H^p$ function. To show this, it is sufficient to show that the boundedness of $S_g:BMOA\to H^p$ implies $g \in H^p$. However, we will get the characterization for the boundedness of $S_g$ from the boundedness of $T_g:H^p\to H^p$.

Again, we invoke the machinery used in Theorem \ref{p>q}. Let again $Z=\{a_k\}$ be a $\rho$-lattice in the Bergman metric. By Theorem 5.29 of \cite{ZhuHol}, if $b>n$, then the function
$$F_t(z)=\sum_k \lambda_k r_k(t)\left(\frac{1-|a_k|^2}{1-\langle z,a_k\rangle}\right)^b$$
belongs to $BMOA$, if the sequence $\{\lambda_k\}$ is chosen so that
$$d\mu_{\lambda}=\sum_k |\lambda_k|^2 (1-|a_k|^2)^n d\delta_{a_k}$$
is a Carleson measure. In the language of tent spaces, this means that $\{\lambda_k\} \in T^{\infty}_2$, or $\{\lambda_k^2\} \in T^\infty_1$. This is a well-known fact about tent spaces, see Section 2.6 of \cite{MPPW} for the formulation on the unit ball.

We assume boundedness of $S_g:BMOA\to H^p$, test with $S_gF_t$, and estimate as we did leading up to \eqref{equ1}, to obtain

\begin{equation}\label{equinf}
	\int_{\mathbb{S}_n} \left( \sum_{{a_k}\in \Gamma(\zeta)}    |\lambda_k|^2 |a_k|^4 |g(a_k)|^2  \right)^{p/2}  \,d\sigma(\zeta) \lesssim \|S_g\|^p \,\|\lambda\|^p_{T^\infty_2}.
\end{equation}
Now, writing $\nu_k=|a_k|^2 |g(a_k)|$, we aim to show that $\{\nu_k\} \in T^p_\infty$. This is the same as showing that $\{\nu_k^{2/s}\}\in T^{ps/2}_\infty$. We choose $s$ big enough so that $ps/2>1$ and $s>1$ hold true. This means that $\{\nu_k^{2/s}\}$ belongs to the dual of $T^{(ps/2)'}_1$. According to the formula in the beginning of the proof of Proposition 6 in \cite{MPPW}, we can factorize $$T^{(ps/2)'}_1=T^{(ps/2)'}_\infty \cdot T^\infty_1.$$ Therefore, we look at 
$$\sum_k \tau_k \lambda_k^2 \nu_k^{2/s} (1-|a_k|^2)^n\asymp \int_{\mathbb{S}_n}\left(\sum_{a_k \in \Gamma(\zeta)}\tau_k\lambda_k^2 \nu_k^{2/s}\right)d\sigma(\zeta), $$
where $\{\tau_k\} \in T^{(ps/2)'}_\infty$.

We now make the split $\lambda_k^2=\lambda_k^{2/s'}\lambda_k^{2/s}$, observe that $\{\lambda_k^{2/s'}\}\in T^\infty_{s'}$, and conclude that if $\alpha_k=\tau_k \lambda_k^{2/s'}$, then $\{\alpha_k\}\in T^{(ps/2)'}_{s'}$.

With these preparations we do a Hölder's inequality with $s'$ and $s$, and arrive at
$$\int_{\mathbb{S}_n}\left(\sum_{a_k \in \Gamma(\zeta)}\tau_k\lambda_k^2 \nu_k^{2/s}\right)d\sigma(\zeta)\leq \int_{\mathbb{S}_n}\left(\sum_{a_k \in \Gamma(\zeta)} \alpha_k^{s'}\right)^{1/s'}\left(\sum_{a_k \in \Gamma(\zeta)} \lambda_k^2 \nu_k^2 \right)^{1/s}d\sigma(\zeta).$$

Finally, we do another Hölder with $(ps/2)'$ and $ps/2$, so that we have proven
$$\sum_k \tau_k \lambda_k^2 \nu_k^{2/s} (1-|a_k|^2)^n\leq \|\alpha\|_{T^{(ps/2)'}_{s'}}\left(\int_{\mathbb{S}_n} \left( \sum_{{a_k}\in \Gamma(\zeta)}    \lambda_k^2 \nu_k^2 \right)^{p/2}  \,d\sigma(\zeta)\right)^{2/ps}.$$

Since $$\|\alpha\|_{T^{(ps/2)'}_{s'}}\lesssim \|\tau\|_{T^{(ps/2)'}_\infty}\cdot \|\lambda^{2/s'}\|_{T^\infty_{s'}},$$
we return to the estimate \eqref{equinf}, and see that 

$$\sum_k \tau_k \lambda_k^2 \nu_k^{2/s} (1-|a_k|^2)^n \lesssim \|\tau\|_{T^{(ps/2)'}_\infty}\cdot \|\lambda^{2}\|_{T^\infty_{1}}^{1/s'}\cdot\|\lambda^2\|_{T^\infty_1}^{1/s}\|S_g\|^{2/s}.$$

By factorization, we finally get

$$\sum_k \tau_k \lambda_k^2 \nu_k^{2/s} (1-|a_k|^2)^n\lesssim \|\tau \lambda^2\|_{T^{(ps/2)'}_1}\|S_g\|^{2/s}.$$

Therefore, $\{\nu_k^{2/s}\}$ must belong to $T^{ps/2}_\infty$, with the norm bound $\|\nu^{2/s}\|_{T^{(ps/2)}_\infty}\lesssim \|S_g\|^{2/s}$. This is equivalent to the membership of $\{\nu_k\}$ in $T^p_\infty$ and $\|\nu\|_{T^p_\infty}\lesssim \|S_g\|$. We can use the same reasoning as in  the end of the proof of Theorem \ref{p>q} to see how this leads to $g \in H^p$ with the appropriate norm bound.
\end{proof}

\begin{rk}
Regarding the atomic decomposition used in the proofs of the main results. In order for it to be a true atomic decomposition, meaning that every element of the corresponding space can be obtained this way, one needs a refinement of the given Bergman lattice. However, as noted in \cite{ZhuHol}, the variant used in the present paper is enough to produce a function in the corresponding space.
\end{rk}

\begin{rk}
The proof goes through with a much easier factorization, if $p>2$. Given that the case $p=2$ is usually the easiest by far, we find it surpising that the proof above requires so much extra effort even in that case.
\end{rk}


\end{document}